\documentclass[11pt]{amsart}
\usepackage[utf8]{inputenc}

\usepackage{amssymb, amsmath, amsthm}
\usepackage{tikz, tikz-cd}
\usepackage{graphicx}
\usepackage{enumerate}
\usepackage{comment}
\usepackage{bbold}
\usepackage{subcaption}
\usepackage{pinlabel}
\usepackage{hyperref}
\hypersetup{
    colorlinks=true,
    linkcolor=purple,
    citecolor=teal,      
    urlcolor=purple,
}
\usepackage{thmtools}
\usepackage{thm-restate}
\usepackage{bm}
\usepackage[all,cmtip]{xy}
\usepackage{cleveref}
\usepackage{overpic}

\usepackage[bordercolor=gray!20,backgroundcolor=blue!10,linecolor=blue!50,textsize=footnotesize,textwidth=0.8in]{todonotes}
\theoremstyle{definition}
\newtheorem{theorem}{{Theorem}}[section]

\newcommand{\hfhat}{\widehat{\textit{HF}}}
\newcommand{\HFKhat}{\widehat{\textit{HFK}}}
\newcommand{\HFKminus}{\textit{HFK}^{-}}
\newcommand{\lossminus}{\mathfrak{L}(L)}

\newcommand{\lossl}{\widehat{\mathfrak{L}}(L)}

\newcommand{\cc}{\mathbf{c}}

\newcommand{\xx}{\mathbf{x}}

\title{A combinatorial description of the LOSS Legendrian knot invariant}

\author[D. He]{Dongtai He}
\address {Shanghai Center for Mathematical Science, Fudan University, Shanghai, China 200438}
\email{dongtaihe@fudan.edu.cn}

\author[L. Truong]{Linh Truong}
\address {Department of Mathematics, Ann Arbor, MI 48103}
\thanks{LT was partially supported by NSF grant DMS-2005539.}
\email{tlinh@umich.edu}

\begin{document}

\begin{abstract} In this note, we observe that the hat version of the Heegaard Floer invariant of Legendrian knots in contact three-manifolds defined by Lisca-Ozsv\'ath-Stipsicz-Szab\'o can be combinatorially computed. We rely on Plamenevskaya's combinatorial description of the Heegaard Floer contact invariant. 
\end{abstract}
\maketitle

\section{Introduction} 
Since the construction of Heegaard Floer invariants introduced by Ozsv\'ath and Szab\'o \cite{OS}, many mathematicians have been studying its computational aspects. Although the definition of Heegaard Floer homology involves analytic counts of pseudo-holomorphic disks in the symmetric product of a surface, certain Heegaard Floer homologies admit a purely combinatorial description. Sarkar and Wang \cite{SarkarWang} showed that any closed, oriented three-manifold admits a nice Heegaard diagram, in which the holomorphic disks in the symmetric product can be counted combinatorially. Using grid diagrams to represent knots in $S^3$, Manolescu, Ozsv\'ath, and Sarkar gave a combinatorial description of knot Floer homology \cite{MOS}. Given a connected, oriented four-dimensional cobordism between two three-manifolds, under additional topological assumptions, Lipshitz, Manolescu, and Wang \cite{LMW} present a procedure for combinatorially determining the rank of the induced Heegaard Floer map on the hat version. 

For three-manifolds $Y$ equipped with a contact structure $\xi$, Ozsv\'ath and Szab\'o \cite{OS-contact} associate a homology class $\cc(\xi) \in \hfhat(-Y)$ in the Heegaard Floer homology which is an invariant of the contact manifold. In \cite{Plamenevskaya}, Plamenevskaya provides a combinatorial description of the Heegaard Floer contact invariant \cite{OS-contact}, by applying the Sarkar-Wang algorithm to the Honda-Kazez-Mati\'c description \cite{HKM} of the contact invariant. 

\begin{theorem}[\cite{Plamenevskaya}] \label{thm:combinatorialcontact}
Given a contact 3-manifold $(Y, \xi)$, the Heegaard Floer contact invariant $\cc(\xi) \in \hfhat(-Y)$ can be computed combinatorially. 
\end{theorem}

A Legendrian knot $L$ in a contact 3-manifold $(Y, \xi)$ is a knot which is everywhere trangent to the contact plane field $\xi$. 
Lisca, Ozsv\'ath, Stipsicz, and Szab\'o define invariants for Legendrian knots inside a contact 3-manifold \cite{loss}, colloquially referred to as the LOSS invariants.  The definition uses open book decompositions and extend the Honda-Kazez-Mati\'c interpretation of the contact invariant. The LOSS invariants are the homology class of a special cycle in the knot Floer homology of the Legendrian knot (see Theorem~\ref{thm:loss} for details). It is natural to ask whether the LOSS invariants admit a combinatorial description. 

In the case of Legendrian knots inside the standard tight contact 3-sphere, Ozsv\'ath, Szab\'o, and Thurston define the GRID invariants \cite{OST}, which are special classes in the knot Floer homology of the Legendrian. The GRID invariants are inherently combinatorial, defined using grid diagrams.  In comparison, by relying on open book decompositions, the LOSS invariants are defined for Legendrian knots in arbitrary contact 3-manifolds, encompassing much greater generality than their GRID counterparts.  

Baldwin, Vela-Vick, and Vertesi \cite{BVVV} prove that the Legendrian invariants  and the GRID invariants  in knot Floer homology agree for Legendrian knots in the tight contact 3-sphere. They introduce an invariant, called BRAID, for transverse knots and Legendrian knots  $L$ inside a tight contact 3-sphere $(S^3, \xi_\text{std})$, which recovers both the LOSS and GRID invariants. 

For Legendrian and transverse links in universally tight lens spaces, Tovstopyat-Nelip \cite{lev} uses grid diagrams to define invariants which generalize the GRID invariants of \cite{OST} and agree with the BRAID and LOSS invariants defined in \cite{BVVV} and \cite{loss}. These invariants are combinatorial, but the contact three-manifolds are limited to universally tight lens spaces. 

We study the question of whether the Legendrian and transverse link invariants in arbitrary contact 3-manifolds of \cite{loss} admit a combinatorial description. In the case of the hat version of the LOSS invariant, we adapt Plamenevskaya's \cite{Plamenevskaya} proof of Theorem~\ref{thm:combinatorialcontact} to obtain a combinatorial description of $\lossl$. 
\begin{theorem}\label{thm:main}
Given a contact 3-manifold $(Y, \xi)$ and a Legendrian knot $L \subset Y$, the LOSS invariant $\lossl$ can be computed combinatorially.
\end{theorem}
\noindent
In Theorem~\ref{thm:combinatorial}, we find a doubly-pointed Heegaard diagram for $(-Y, L)$ which is a nice diagram in the sense of Sarkar-Wang \cite{SarkarWang}. Nice diagrams are Heegaard diagrams $(S, \alpha, \beta, w)$ in which any region of $S \setminus (\alpha \cup \beta)$ not containing the basepoint $w$ is either a bigon or a square. Nice diagrams have the property that the counts of pseudo-holomorphic disks which appear in the Heegaard Floer differential maps are combinatorially determined. Theorem~\ref{thm:main} then follows.

It is known that the LOSS invariants also define transverse knot invariants via Legendrian approximation. Thus, it follows from Theorem~\ref{thm:main} that the hat version of the LOSS invariant of transverse knots can also be computed combinatorially. 

\section{The Legendrian LOSS invariants}

Suppose that $L \subset (Y, \xi)$ is a Legendrian knot in a contact three-manifold. Consider an open book decomposition $(\Sigma, \phi)$ compatible with $\xi$ containing $L$ on a page. Recall that Honda-Kazez-Mati\'c construct a Heegaard diagram for $-Y$ associated to $(\Sigma, \phi)$. This Heegaard diagram gives rise to an explicit description of the Heegaard Floer contact invariant $c(Y, \xi)\in\hfhat (-Y)$ originally defined in \cite{OS-contact}, by identifying the contact invariant with the homology class of the cycle $\cc = \{c_1, \dots, c_{2g}\}$, a $2g$-tuple of intersection points in the Heegaard diagram.

Lisca-Ozsv\'ath-Stipsicz-Szab\'o build on the Honda-Kazez-Matic algorithm to give a doubly-pointed Heegaard diagram $(S, \beta, \alpha, w, z)$ for $(-Y, \xi, L)$. 
By \cite[Lemma 3.1]{loss}, there is a basis of arcs $\{a_1, \dots, a_{2g}\}$ for $\Sigma$ adapted to $L$, meaning that $L \cap a_i= \emptyset$ for $i \geq 2$ and $L$ intersects the arc $a_1$ in a unique transverse point. That the set of arcs $\{a_1, \dots, a_{2g}\}$ forms a basis for $\Sigma$ means that the arcs are disjoint properly embedded arcs on $\Sigma$ such that $\Sigma \setminus \cup_{i=1}^n a_i$ is a disk. We obtain another basis $\{b_1, \dots, b_{2g}\}$ for $\Sigma$, where each arc $b_i \subset \Sigma$ is obtained from $a_i$ by a small isotopy in the direction given by the boundary orientation, and $b_i$ intersects $a_i$ in exactly one transverse intersection point $x_i \in \Sigma$, as in Figure~\ref{fig:hkm_ob}.

\begin{figure}[h]
\centering
\begin{overpic}[abs,unit=1mm,scale=.5]{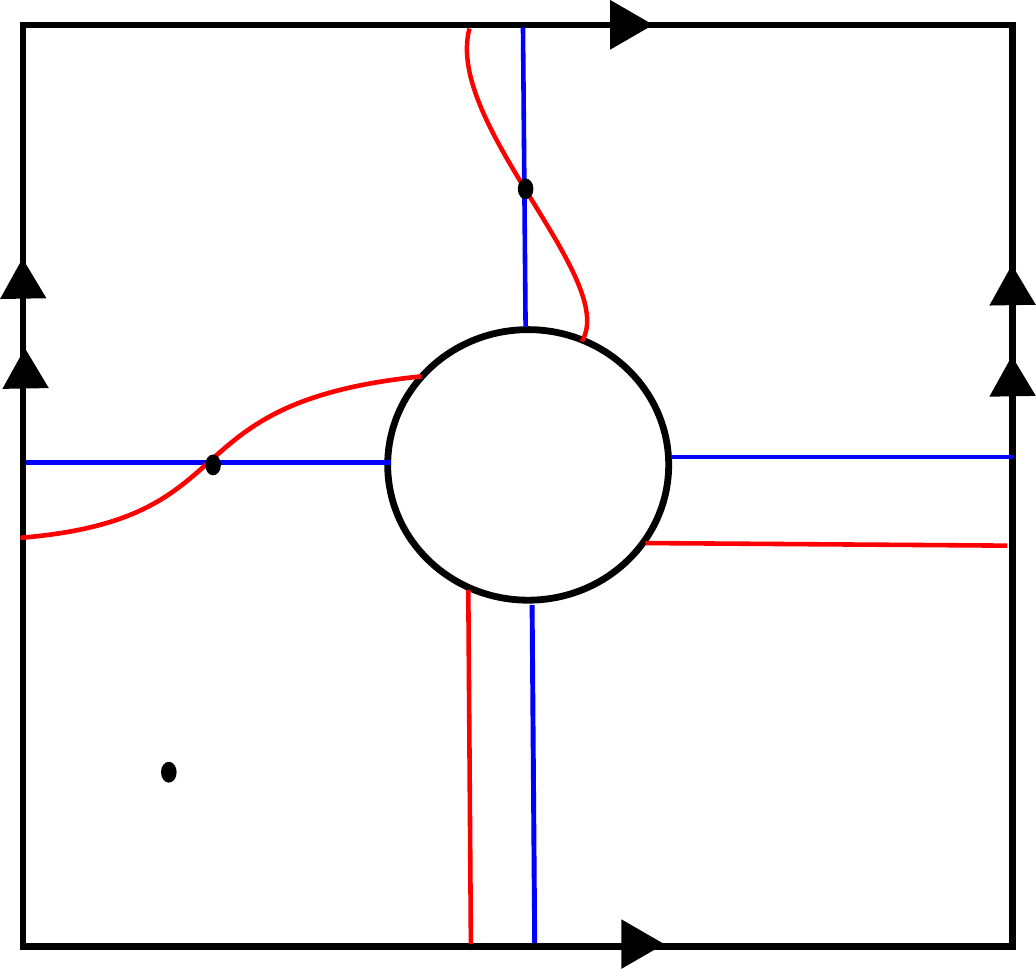}
\put(10,10){$w$}
\put(3,20){$\alpha_1$}
\put(3,28){$\beta_1$}
\put(10,23){$x_1$}
\put(19,3){$\alpha_2$}
\put(28,3){$\beta_2$}
\put(22,39){$x_2$}
\put(33,30){$\partial \Sigma$}
\end{overpic}
\centering
\caption{An example of $(\Sigma, \alpha, \beta, w)$, where $\Sigma$ is a surface of genus 1 with one boundary component, $\alpha=\{\alpha_1, \alpha_2\}$ consists of the two red arcs, and $\beta =\{\beta_1, \beta_2\}$ consists of small isotopic translates of the $\alpha$ arcs. }
\label{fig:hkm_ob}
\end{figure}
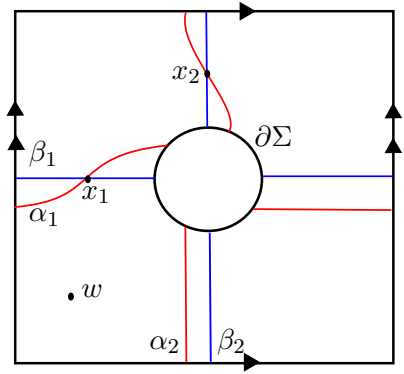

Let: 
\begin{itemize}
\item $S = \Sigma \cup -\Sigma$, or the union of two copies of the fiber surface $\Sigma$ glued along their boundary,
\item $\alpha = \{\alpha_1, \dots, \alpha_{2g}\}$, where $\alpha_i = a_i \cup \overline{a_i}$, where $\overline{a_i} \subset -\Sigma$ is a copy of $a_i$,
\item $\beta = \{\beta_1, \dots,\beta_{2g}\}$, where $\beta_i = b_i \cup \overline{\phi(\beta_i)}$, where $\overline{\phi(b_i)} \subset -\Sigma$ is a copy of $\phi(b_i)$, 
\item the basepoint $w \in \Sigma$ lies outside of the thin regions created by the isotopies between $a_i$ and $b_i$,
\item the  basepoint $z \in \Sigma $ is placed in one of the two thin regions between $a_1$ and $b_1$, determined by the orientation of the Legendrian $L$. See Figure~\ref{fig:loss}.
\end{itemize}
\noindent
In particular, the $(S, \beta, \alpha, w)$ agrees with the Honda-Kazez-Matic singly-pointed Heegaard diagram for $-Y$.

\begin{figure}[h]
\centering
\begin{overpic}[abs,unit=1mm,scale=.5]{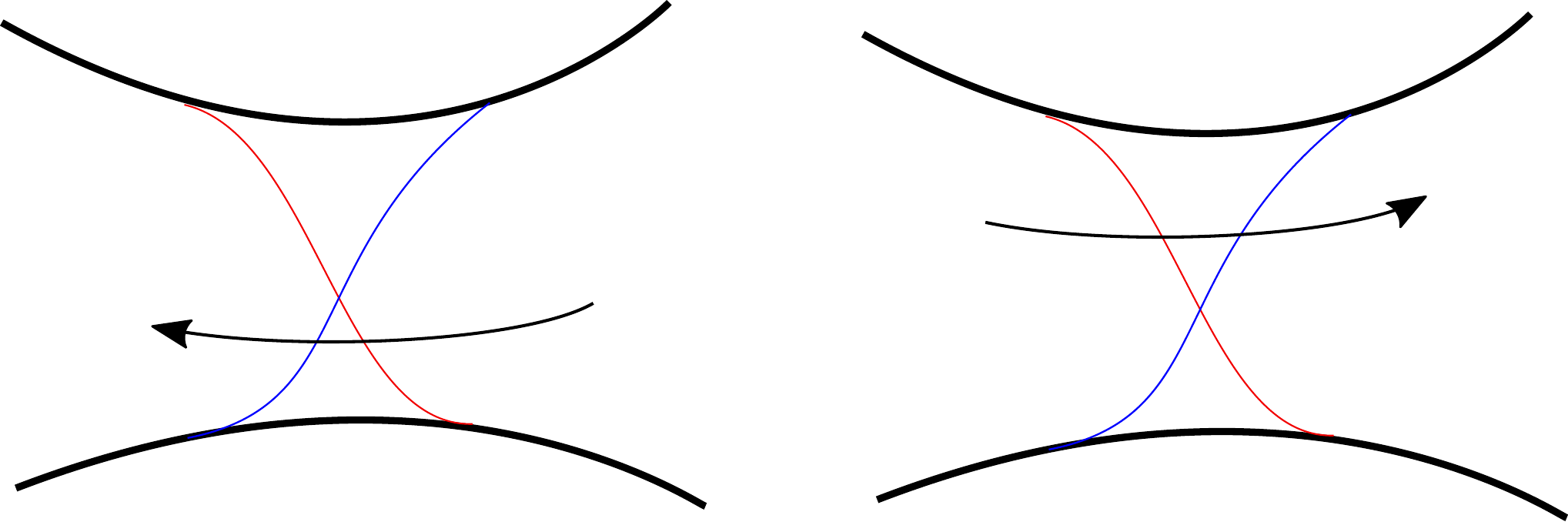}
\put(15,20){$\alpha_1$}
\put(28,20){$\beta_1$}
\put(71,20){$\alpha_1$}
\put(84,20){$\beta_1$}
\put(40,14){$ w$}
\put(21, 10){$ z$}
\put(78, 20){$ z$}
\put(95, 20){$ w$}
\put(6, 12){$ L$}
\put(97, 12){$ L$}
\end{overpic}
\centering
\caption{The basepoint is placed in one of the two regions of the small isotopy between the arcs $\alpha_1$ and $\beta_1$, corresponding to the orientation of the Legendrian knot $L$.} 
\label{fig:loss}
\end{figure}
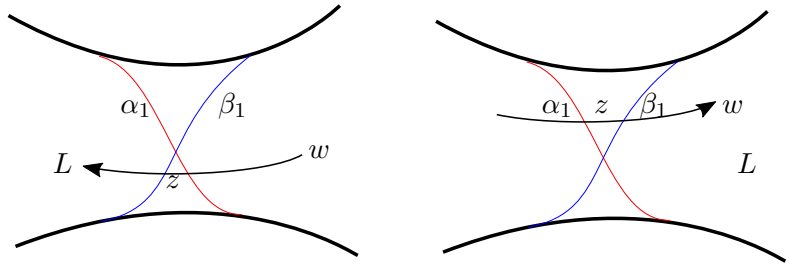

The contact invariant $c(\xi) \in \hfhat(-Y)$ is the homology class of the cycle $\cc = \{x_1, \dots, x_{2g}\}$ in $ \hfhat(S, \beta, \alpha, w)$. The combinatorial description of the contact invariant $c(\xi)$ stated in Theorem~\ref{thm:combinatorialcontact} is a corollary of Theorem~\ref{thm:plam}. 

\begin{theorem}[{\cite[Theorem 2.1]{Plamenevskaya}}] 
\label{thm:plam}
Suppose $(\Sigma, h)$ is an open book decomposition for the contact 3-manifold $(Y, \xi)$. 
There exists an equivalent open book decomposition $(\Sigma, h'')$  for $(Y, \xi)$ such that singly-pointed Heegaard diagram described by Honda-Kazez-Matic for $-Y$ has only bigon and square regions (except for the one polygonal region $D_w$ containing the basepoint $w$). The monodromy $h''$ differs from $h$ by an isotopy; that is, $h'' = \psi \circ h$, where $\psi: \Sigma \to \Sigma$ is a diffeomorphism fixing the boundary and isotopic to the identity. 
\end{theorem}

The LOSS invariant is an invariant of the Legendrian knot $L$ determined by the $2g$-tuple of intersection points $\xx = \{x_1, \dots, x_{2g}\}$ which also defines a cycle in $ \HFKhat(-Y, L, \mathbf{t}_\xi)$.

\begin{theorem}[\cite{loss}] \label{thm:loss} Let $L$ be an oriented, null-homologous Legendrian knot in the closed contact three-manifold $(Y, \xi)$, and let $t_\xi$ denote the spin$^c$ structure on $Y$ induced by $\xi$. Then $\lossl$ and $\lossminus$ are invariants of the Legendrian isotopy class of $L$, where
\begin{itemize}
\item $\lossl $ is defined as the isomorphism class $ [\HFKhat(-Y, L, \mathbf{t}_\xi), [\xx ]]$, where $[\xx]$ is the homology class of the cycle $\xx = \{x_1, \dots, x_{2g}\}$, and 
\item $\lossminus$ as the isomorphism class $ [\HFKminus (-Y, L, \mathbf{t}_\xi), [\xx ]]$.
\end{itemize}
\end{theorem}

We give a combinatorial description of the hat version $\lossl$ of the LOSS invariant.

\begin{theorem} 
\label{thm:combinatorial}
Suppose $(\Sigma, h)$ is an open book decomposition for the contact 3-manifold $(Y, \xi)$, equipped with a basis $\{a_i\}$ of arcs in $\Sigma$ adapted to the Legendrian knot $L \subset Y$. 
There exists an equivalent open book decomposition $(\Sigma, h')$ for $(Y, \xi)$ and a basis $\{a_i''\}$ adapted to $L$ such that the doubly-pointed Heegaard diagram for $(-Y, L)$, described by Lisca-Ozsv\'ath-Stipsicz-Szab\'o, has only bigon and square regions (except for the one polygonal region $D_w$ containing the basepoint $w$). The monodromy $h''$ differs from $h$ by an isotopy; that is, $h' = \psi \circ h$, where $\psi: \Sigma \to \Sigma$ is a diffeomorphism fixing the boundary and isotopic to the identity. 
\end{theorem}
\begin{proof}
Consider the open book decomposition $(\Sigma, h )$ with basis $\{a_i\}$ adapted to $L \subset Y$ as above. Let $S = \Sigma \cup -\Sigma$. 
We follow the steps in the proof of Theorem~\ref{thm:plam} to obtain an equivalent open book decomposition  $(\Sigma, h')$  for $Y$ such that the associated singly-pointed Heegaard diagram $\mathcal{H}' = (S,  \beta', \alpha, w)$ described by Honda-Kazez-Mati\'c for $-Y$ has only bigon and square regions (except for the one polygonal region $D_w$ containing the basepoint $w$). The monodromy $h'$ differs from $h$ by an isotopy; that is, $h' = \phi \circ h$, where $\phi: \Sigma \to \Sigma$ is a diffeomorphism fixing the boundary and isotopic to the identity.

In more details, in the proof of Theorem~\ref{thm:plam}, Plamenevskaya applies the Sarkar-Wang algorithm to modify the Heegaard diagram via a sequence of isotopies. In particular, all of these isotopies occur on the $\beta$ curves and occur in the $-\Sigma$ (which we refer to as the monodromy side of $S$) part of the Heegaard surface $S$. Recall that the second basepoint $z$ is placed in a region between $a_1$ and $b_1$ in $\Sigma$ (which we refer to as the standard side of $S$). Thus, an isotopy of $\beta_i$ occurs entirely in $-\Sigma$ and either:
\begin{enumerate}
\item does not go through the region containing $z$. 
\item does go through the region containing $z$. If the isotopy is a finger move through the region containing $z$, then it divides this region into two pieces, one of which is entirely contained on the monodromy side $-\Sigma$; the basepoint $z$ is chosen to remain in the region that intersects the standard side $\Sigma$.  
\end{enumerate} 
We emphasize that these isotopies never cross the arc $\delta_\alpha$, respectively $\delta_\beta$, connecting $w$ and $z$ in the complement of the $\alpha$ circles, respectively $\beta$ circles, since the arcs $\delta_\alpha$ and $\delta_\beta$ lie on the standard side $\Sigma$ of the surface $S$. The resulting diagram $(S, \alpha, \beta', w, z)$ is a doubly-pointed nice diagram for the Legendrian $L$, since $L$ is Legendrian isotopic to $ \delta_\alpha \cup \delta_\beta$. Thus, the resulting filtered chain complex associated to the doubly-pointed Heegaard diagram $(S', \alpha', \beta', w, z)$ is filtered chain homotopic to the filtered chain complex of $(S, \alpha, \beta, w, z)$ by \cite[Theorem 3.1]{OS04}.

Therefore, we obtain a nice Heegaard diagram by performing this sequence of isotopies in $-\Sigma \subset S$. A composition of these isotopies gives a diffeomorphism $\psi:\Sigma \to \Sigma$ fixing the boundary and isotopic to the identity. The resulting open book decomposition $(\Sigma, h' = \psi \circ h)$ is equivalent to the original open book decomposition  $(\Sigma, h)$. 
\end{proof}

\bibliographystyle{alpha}
\bibliography{bib}

\end{document}